\documentclass{amsart}

\usepackage{enumerate}
\usepackage{amsmath,amsfonts,amssymb,amsthm}
\usepackage{csquotes}
\usepackage{graphicx,amsaddr}

\newtheorem{thm}{Theorem}[section]
\newtheorem{lem}[thm]{Lemma}
\newtheorem{prop}[thm]{Proposition}
\newtheorem{cor}[thm]{Corollary}

\theoremstyle{definition}
\newtheorem{defn}[thm]{Definition}
\newtheorem{exam}[thm]{Example}

\numberwithin{equation}{section}

\title[J. G. Patel]
{A Class of algebras admitting infinitely many norm topologies}

\author{J. G. Patel*}
\address{Dept. of Applied Science and Humanities, Parul University, Limda 391760, Vadodara, Gujarat, India}
\email{jatinkumar.patel33697@paruluniversity.ac.in}

\begin{document}
	
	\subjclass[2020]{46H05, 46H25.}
	
	\keywords{Algebra, Banach Algebra, Norm }
	
\begin{abstract}
Let $\mathcal{A}$ be an algebra, and let $\mathcal{A}^2 =$ span$\{ab : a, b \in \mathcal{A}\}$ be a subalgebra of $\mathcal{A}$. In this paper, we prove that if $\mathcal{A}^2$ has infinite codimension in  $\mathcal{A}$ iff $\mathcal{A}$ has discontinuous square annihilation property (DSAP). In fact, in this case, the algebra
$\mathcal{A}$ admits infinitely many non-equivalent algebra norms.
\end{abstract}
\maketitle
\section{Introduction}
Throughout $\mathcal{A}$ is an (associative) algebra over the complex field $\mathbb{C}$. A norm on $\mathcal{A}$, means it is a linear norm and submultiplicative (i.e., $\| ab \| \leq \|a\| \|b\| \ \ (a,b \in \mathcal{A})$). An equivalent norms means the metric topology induced by these norms same. It is evident that $\mathcal{A}^2 =$ span$\{ab : a, b \in \mathcal{A}\}$ forms a subalgebra of $\mathcal{\mathcal{A}}$. In some cases, it may happen that $\mathcal{A}^2 = \mathcal{A}$.

It is a well-known fact that a vector space admits either a unique linear norm (up to equivalence) or infinitely many linear norms (up to equivalence). So, it is natural to ask for the case of algebra. It is surprising that the behaviour of submultiplicative norm on algebra is extremely different. From the following references~[\cite{DL:97}, \cite{DP:22(a)}]. So far there are three types of classification of an algebra norm. These are (i) there is an algebra which has no any algebra norm, e.g. $C(\mathbb{R})$ The set of all continuous functions on $\mathbb{R}$ with usual operations and pointwise multiplication, (ii) there is an algebra which has a unique algebra norm, e.g. $B(H)$ (iii) there is an algebra which has infinitely many algebra norms, e.g., the disc algebra $A(\mathbb{D})$. From these classification it naturaly leads the following question: What about an algebra having only finitely many more than one norms !!!. Also this question was asked in author's doctorate thesis. If such kind of an algebra exist, then the codimension $\mathcal{A}/ \mathcal{A}^2$ is finite can be conclude from the Corollary~\ref{Infinitely Many Norms}.

Dales and Loy \cite{DL:97} gave a nice example of an algebra with a finite-dimensional radical, and built two different algebra norms on it. Encouraged by them, the current paper generalizes the result further with a proof of the fact that not only two but actually an infinite number of inequivalent algebra norms can be constructed on such an algebra.

\section{Main results}
\noindent Throught $\mathcal{A}$ is an infinite dimensional algebra.

\begin{defn}
	Let $(\mathcal{A}, \| \cdot \|)$ be a normed algebra. An algebra $\mathcal{A}$ has the discontinuous square annihilation property (DSAP) if there exists a   discontinuous linear functional $\varphi$ on $\mathcal{A}$ such that $\mathcal{A}^2 \subseteq $ ker$\varphi$.
\end{defn}

\begin{prop}
	Let $(\mathcal{A}, \| \cdot \|)$ be a normed algebra. Then  if $\mathcal{A}$ right (left) unital or $\mathcal{A}$ has bounded approximate identity. Further, if 
\end{prop}

\begin{proof}
	In both the cases $\mathcal{A} = \mathcal{A}^2$. Hence, every linear functional having $\mathcal{A}^2 \text{ker}\phi$ is identically zero. Now, for converse consider $c_{00}$ with pointwise linear, scalar, and multiplication operations. Then $(c_{00}, \| \cdot \|_1)$ is normed algebra. Now it's clear that $c_{00}^2= c_{00}$ but it neither have right (left) identity nor bounded approximate identity.
\end{proof}

\begin{lem}
	Let $(\mathcal{A}, \| \cdot \|)$ be a normed algebra. Let $ \varphi$ be a linear functional on $\mathcal{A}$ such that the $\mathcal{A}^2 \subseteq $ker$ (\varphi)$. For each $a \in \mathcal{A}$, define $$p(a) = \| a \| + |\varphi(a)| .$$ Then
	\begin{enumerate}
		\item $p(\cdot)$ is an algebra norm on $\mathcal{A}$.
		\item  $\varphi$ is continuous on $\mathcal{A}$ iff $p(\cdot) \cong \| \cdot \|$.
	\end{enumerate}
\end{lem}

\begin{proof}
	The proofs are very easy for both cases.
\end{proof}

\begin{thm}
	Let $(\mathcal{A}, \| \cdot \|)$ be a normed algebra. The codimension of $\mathcal{A}^2$ in $\mathcal{A}$ is infinite iff A has DAP.
\end{thm}

\begin{proof}
	Let $(\mathcal{A}, \| \cdot \|)$ be a normed algebra. Since $\mathcal{A}^2$ has infinite co-dimension in $\mathcal{A}$, there exists a countably infinite linearly independent subset $L = \{a_1, a_2,a_3, \ldots\}$ of $\mathcal{A}$ such that $\mathcal{A}^2 \cap L = \phi$ and $\| a \| =1 \; (a \in L)$. 
	
	Let $D$ be a basis of $\mathcal{A}^2$. Now, consider the (unique) linear map $\varphi : \mathcal{A} \longrightarrow \mathbb{C}$ such that
	\[
	\varphi(a)=
	\begin{cases}
		k&  (\text{if } a=a_{k}  \in L \});\\
		1 & (\text{if } a \in B \setminus (L \cup D)); \\
		0 & (\text{if } a \in D \}).
	\end{cases}
	\]
	Now $a \in \mathcal{A}$, define
	$$p(a)=\|a\|+ |\varphi(a)|.$$
	Then $(\mathcal{A}, p( \cdot ))$ is a normed algebra.
	Clearly, each $p(\cdot)$ is a linear norm. Let $a, b \in \mathcal{A}$. Then $p(ab)  = \|ab\| + |\varphi(ab| \leq p(a)p(b)$.Thus $p(\cdot)$ is submultiplicative. Now, we \textit{claim} that $\| \cdot \|$ and $p(\cdot)$are non-equivalent on $\mathcal{A}$. Let $a_{k} \in L \ (k \in \mathbb{N})$. Then, for $k \geq 2, \; p(a_k) = 1+k$ and $\|a_k\| = 1$. Thus $\| \cdot \|$ is not equivalent to $p(\cdot)$. Hence $p(\cdot)$ is inequivalent to $\| \cdot \|$.
	Conversely assume that $\mathcal{A}$ has DAP, so there is discontinuous linear functional $\phi: \mathcal{A} \longrightarrow \mathbb{C}$ such that $\mathcal{A}^2 \subset \text{Ker} \phi$. Now suppose the codimension of $\mathcal{A}^2$ in $\mathcal{A}$ is finite, the embeeding map $\tilde{\phi}: \mathcal{A}/ \mathcal{A}^2 \longrightarrow \mathbb{C}$ defined by $\tilde{\phi}(x+ \mathcal{A}^2)= \phi(x)$ is continuous. Define another map say $\psi: \mathcal{A} \longrightarrow \mathcal{A}/\mathcal{A}^2$, $$\psi(x)= x+ \mathcal{A}^2$$ is continuous as ker$\phi= \mathcal{A}^2$. Hence $\phi$ is continuous as $\phi= \tilde{\phi} \circ \psi$.
\end{proof}
 
 \begin{cor}{\label{Infinitely Many Norms}}
 	Let $\mathcal{A}$ be a normed algebra, and let the codimension of $\mathcal{A}^2$ in  $\mathcal{A}$ is infinite. Then $\mathcal{A}$ admits infinitely many algebra norm.
 \end{cor}
 
 \begin{proof} Let $(\mathcal{A}, \| \cdot \|)$ be a normed algebra. Since $\mathcal{A}^2$ has infinite co-dimension in $\mathcal{A}$, there exists a countably infinite linearly independent subset $L$ of $\mathcal{A}$ such that $\mathcal{A}^2 \cap L = \phi$ and $\| a \| =1 \; (a \in L)$. For each $n \in \mathbb{N}$, choose $L_n = \{ a_{n1}, a_{n2}, \ldots \} \subset L$ such that: 
 	\begin{enumerate}
 		\item Each $L_n$ is infinite;
 		\item $L_n \cap L_m = \phi \ (n \neq m)$;
 		\item  $L = \bigcup_{n=1}^{ \infty} L_n$.
 	\end{enumerate}
 	Let $D$ be a basis of $\mathcal{A}^2$. Let $B_n$ be a (Hamel) basis of $\mathcal{A}$ such that $L_n \cup D \subset B_n$ for each $n \in \mathbb{N}$. Now, consider the (unique) linear map $\varphi_n : \mathcal{A} \longrightarrow \mathbb{C}$ such that
 	\[
 	\varphi_n(a)=
 	\begin{cases}
 		k&  (\text{if } a=a_{nk}  \in L_n \setminus \{a_{n1}\});\\
 		1 & (\text{if } a \in B_n \setminus (L_n \cup D)); \\
 		0 & (\text{if } a \in D \cup \{a_{n1}\}).
 	\end{cases}
 	\]
 	Now for $a \in \mathcal{A}$, define
 	$$p_n(a)=\|a\|+ |\varphi_n(a)|.$$
 	Then $(\mathcal{A}, p_n( \cdot ))$ is a normed algebra.
 	Clearly, each $p_n(\cdot)$ is a linear norm. Let $a, b \in \mathcal{A}$. Then $p_n(ab)  = \|ab\| \leq p_n((a, \alpha))p_n((b, \beta))$ because $ab \in \mathcal{A}^2$ and hence $\varphi_n(ab) = 0$. Thus each $p_n(\cdot) \in N(\mathcal{B})$. Now, we \textit{claim} that these norms are non-equivalent on $\mathcal{A}$. Let $m < n$ and $g_k = a_{mk} \ (k \in \mathbb{N})$. Then, for $k \geq 2, \; p_m((a_k, 0)) = 1+k$ and $p_n((a_k, 0)) \leq 2$ because $\varphi_n(a_{mk}) = 0 \text{ or 1}$. Thus we have proved our claim.
 \end{proof}
 
\begin{cor}
	The norms $\| \cdot \|$ and $p_n(\cdot)$ are inequivalent even $(\mathcal{A}, \| \cdot \|)$ is a Banach algebra.
\end{cor}

\begin{thm}
	Let $\mathcal{A}$ be an algebra having only finitely many norms.
	\begin{enumerate}
		\item Then $\mathcal{A}$ has maximum and minimum norms. 
		\item Then $\mathcal{A}$ has atmost one complete norm.
		\item The codimension of $\mathcal{A}^2$ in $\mathcal{A}$ is finite.
	\end{enumerate}
	 
\end{thm}

\begin{proof}
	(i) Let  $\| \cdot \|_i \ \ (1 \leq i \leq n)$ are norms on $\mathcal{A}$, then they form a sequence like $ \| \cdot \|_i \leq \ldots \leq  \| \cdot \|_j  \leq \ldots \leq \| \cdot \|_k$ for some $1 \leq i, j, k, \leq n$, where $\| \cdot \|_i$ is the smallest and $\| \cdot \|_k$ is the biggest norm.

	(ii) It follows from two norm theorem.
	
	(iii) It follows from Corollary~\ref{Infinitely Many Norms}.
\end{proof}

\section{Examples}
\begin{exam}
	Let $\mathcal{A} = \ell^2$ with pointwise operations. Now $(\ell^2)^2= \ell^1$. The codimension of $\ell^2/ \ell^1$ is infinite. So, by the Corollary~ \ref{Infinitely Many Norms}, $\mathcal{A}$ has infinitely many norms.
\end{exam}

\begin{exam}
	Let $A = A(\mathbb{D})$ be a disc algebra. Then it shown that it has infinitely many norms. So, this example says that the converse of the Corollary~\ref{Infinitely Many Norms} is not true as $A(\mathbb{D})$ is unital algebra.
\end{exam}

\begin{exam}
	Let $(\mathcal{A}, \| \cdot \|)$ be an algebra. Then $\mathcal{A} \times \mathbb{C}$ is an algebra with pointwise operations and multiplication $(a, \alpha) \times (b, \beta) = (ab,0)$ for $((a,\alpha), (b, \beta) \in \mathcal{A} \times \mathbb{C})$. It was shown that $\mathcal{A} \times \mathbb{C}$ has infinitely many norms.
\end{exam}

\begin{exam}
	Let $\mathbb{F}[x]= \{p(x)= a_0+ a_1x+ \ldots+ a_n x^n : a_0, a_1, \ldots+ a_n \in \mathbb{F}\}$. Consider $\mathcal{A}= \{p(x) \in \mathbb{F}[x] : p^{i}(0)=0, (0 \leq i \leq n-1) \}$. Then $\mathcal{A}$ is an algebra with the usual linear and scalar multiplication, and pointwise multiplication opearions. Now we can easily conclude that if $p(x) \in \mathcal{A}$, then $p(x)=x^n q(x)$, for some $q(x) \in \mathbb{F}[x]$. So, $\mathcal{A}^2= x^n \mathcal{A}$. Hence the co-dimension of $\mathcal{A}^2$ in $\mathcal{A}$ is n, as $\{1,x, x^2, \ldots, x^{n-1}\}$ is basis of $\mathcal{A}/\mathcal{A}^2$.
\end{exam}

\end{document}